\documentclass{amsart}   	

\usepackage{color}
\usepackage[usenames,dvipsnames,svgnames,table]{xcolor}

\usepackage{graphicx}				

\usepackage{amssymb}
\usepackage{float}

\newtheorem{X}{X}[section]
\newtheorem{theorem}[X]{Theorem}
\newtheorem{lemma}[X]{Lemma}

\numberwithin{equation}{section}

\renewcommand{\le}{\ensuremath{\leqslant} }
\renewcommand{\ge}{\ensuremath{\geqslant} }

\newcommand{\cA}{\ensuremath{\mathcal{A}} }
\newcommand{\cP}{\ensuremath{\mathcal{P}} }

\newcommand{\N}{\ensuremath{\mathbb{N}}}

\newcommand{\genlegendre}[4]{%
  \genfrac{(}{)}{}{#1}{#3}{#4}%
  \if\relax\detokenize{#2}\relax\else_{\!#2}\fi
}
\newcommand{\legendre}[3][]{\genlegendre{}{#1}{#2}{#3}}

\title{Symmetric primes revisited}
\author{William Banks}
\address{Department of Mathematics, University of Missouri, Columbia, MO
65211, USA}
\email{bankswd@missouri.edu}
\author{Paul Pollack}
\address{Department of Mathematics, University of Georgia, Athens, GA 30602, USA}
\email{pollack@uga.edu}
\author{Carl Pomerance}
\address{Mathematics Department, Dartmouth College, Hanover, NH 03755, USA}
\email{carl.pomerance@dartmouth.edu}

\newenvironment{dedication}
        {\vspace{6ex}\begin{quotation}\begin{center}\begin{em}}
        {\par\end{em}\end{center}\end{quotation}}

\begin{document}
\maketitle
\vskip-1cm
\begin{dedication}
{In memory of Peter Fletcher (1939--2019)}
\end{dedication}

\begin{abstract}
A pair of odd primes is said to be {\it symmetric} if
each prime is congruent to one modulo their difference.
A theorem from 1996 by Fletcher, Lindgren,
and the third author provides an upper bound on the
number of primes up to $x$ that belong to a symmetric pair.
In the present paper, that theorem is improved to what is likely to be
the best possible result.  We also establish that there exist
infinitely many symmetric pairs of primes.
In fact, we show that for every
integer $m\ge 2$ there is a string of $m$ consecutive primes,
any two of which form a symmetric pair.
\end{abstract}

\vskip1cm

\section{Introduction}
A pair of distinct odd primes $\{p,q\}$ is said to be a
\emph{symmetric pair} if
$$\gcd(p-1,q-1)=|p-q|.$$  For example, every twin prime pair $\{p,p+2\}$
is a symmetric pair.
We say that a prime is \emph{symmetric}
if it is a member of some symmetric pair; otherwise,
we say that it is \emph{asymmetric}.

Symmetric primes arise naturally when ruminating on a common textbook proof of Gauss's quadratic reciprocity law (QRL). Consider the rectangle $S$ with $(0,0)$ and $(p/2,q/2)$ as opposite corners, and let $l$ be the diagonal joining those corners. Let $S(q,p)$ (resp., $S(p,q)$) be the number of interior lattice points below (resp., above) $l$. Eisenstein, in his version of Gauss's third QRL proof, showed that $$\legendre{q}{p}=(-1)^{S(q,p)},\quad\text{and}\quad\legendre{p}{q}=(-1)^{S(p,q)}.$$
Since $\ell$ has no interior lattice points, $S(q,p)+S(p,q)$ is the total number of lattice points interior to $S$, which is $\frac{p-1}{2} \frac{q-1}{2}$; the law of quadratic reciprocity follows immediately. It is shown in \S2 of Fletcher~\emph{et al.}~\cite{FLP} that $\{p,q\}$ is symmetric precisely when $S(q,p)=S(p,q)$.

Most primes are asymmetric;
this is a consequence of \cite[Theorem~3.1]{FLP}, which asserts
that the number $S(x)$ of symmetric primes
$p\le x$ is $O(\pi(x)/(\log x)^{0.027})$. It is conjectured in \cite{FLP}
that the exponent $0.027$ can be
improved to $\eta+o(1)$, where
$$
\eta:=1-\frac{1+\log\log 2}{\log 2}=0.08607\cdots\,.
$$
In this note we prove the conjecture.
\begin{theorem}
\label{thm:main1}
For all large $x$, we have
$$
S(x)\le\frac{\pi(x)}{(\log x)^\eta}(\log\log x)^{O(1)}.
$$
\end{theorem}

The constant $\eta$ appears at several places in the literature.  An early
appearance is in connection with the Erd\H os multiplication table problem
where, thanks to the work of Erd\H os, Tenenbaum, and Ford, we now know
that the number $M(N)$ of distinct entries in the $N\times N$
multiplication table is $N^2(\log N)^{-\eta}(\log\log N)^{O(1)}$.
(In fact, Ford \cite{F} has further shown that if the implied constant $O(1)$
is replaced with $-3/2$, the resulting expression has the same
magnitude as $M(N)$.) A more recent appearance of $\eta$ occurs
in Chow and Pomerance \cite{CP}, where the odd legs in integer-sided right triangles with prime hypotenuse are considered. (The present note
uses some techniques from \cite{CP}.)


It was left as an open problem in \cite{FLP} to prove that there are infinitely many symmetric primes.  The next theorem uses an old result of Heath-Brown~\cite{HB}
together with the framework of the recent results of Zhang, Maynard, Tao,
\emph{et al.}\ on small gaps between primes.

\begin{theorem}
\label{thm:main2}
For every integer $m\ge 2$, there exists a string of $m$ consecutive primes,
any two of which form a symmetric pair.
\end{theorem}

Of course, Theorem \ref{thm:main2} implies the infinitude of symmetric primes. Our proof of Lemma \ref{lem:two} below, in conjunction with the methods of \cite{M2,P1,P2}, could be developed to prove that $S(x) \gg \pi(x)/(\log x)^{49}$.  Comparing this lower bound with the upper bound of Theorem \ref{thm:main1}, it is tempting to conjecture that $S(x) = \pi(x)/(\log{x})^{c+o(1)}$, as $x\to\infty$, for some positive constant $c$. In \cite{FLP} a heuristic argument is presented suggesting that this conjecture holds with $c=\eta$; that is, the inequality of Theorem \ref{thm:main1} is actually an equality.

%

In Sections 2 and 3 we prove the theorems.
In Section 4 we present some new computations of symmetric primes.
In Section 5 we close with a few problems of a somewhat different nature.

\section{The proof of Theorem \ref{thm:main1}}

Let $\omega(n)$ denote the number of distinct primes that divide $n$,
 and let $\Omega(n)$ denote the
number of prime factors of $n$ counted with multiplicity. Let $P^+(n)$
denote the largest prime factor of $n>1$, and put $P^+(1)=0$.

Let $S_1(x)$ denote the number of primes $p\le x$ with
 $P^+(p-1)\le x^{1/\log\log x}$.
Since the number of integers $n\le x$ with $P^+(n)\le x^{1/\log\log x}$ is
$O(x/(\log x)^2)$  (see de Bruijn \cite[Eq.\ (1.6)]{dB})
it follows that $S_1(x)=O(\pi(x)/\log x)$.

Next, let $S_2(x)$ denote the number of
primes $p\le x$ with $P^+(p-1)>x^{1/\log\log x}$ and $\Omega(p-1)>L$, where
$L=\lfloor (1/\log 2)\log\log x\rfloor$.  We claim that
\begin{equation}
\label{S2}
S_2(x)\le\frac{\pi(x)}{(\log x)^\eta}(\log\log x)^{O(1)}.
\end{equation}
For any prime counted by $S_2(x)$, write $p=ar+1$, where
$r=P^+(p-1)>x^{1/\log\log x}$.    For any fixed choice of
$a<x^{1-1/\log\log x}$, the number of primes $r\le x/a$ with $ar+1$ prime is
(by Brun's method; see \cite[Eq.\ (6.1)]{HR}) at most
$$
\frac x{a(\log x)^2}(\log\log x)^{O(1)}.
$$
We sum this expression over $a$ assuming $\Omega(a)\ge L$.
For $L\le\Omega(a)\le 1.9\log\log x$ we use \cite[Theorem~08]{HT}, finding that
$\sum1/a\le(\log x)^{1-\eta}(\log\log x)^{O(1)}$; this
is consistent with our goal \eqref{S2}.
For larger values of $\Omega(a)$ we use \cite[Exercise~05]{HT}, getting $\sum1/a\ll (\log x)^{0.69}$.
This establishes the claim \eqref{S2}.

To finish the proof, we bound the number of symmetric
primes $p\le x$ with $P^+(p-1)>x^{1/\log\log x}$ and $\Omega(p-1)\le L$.
For any such prime, write $p=ar+1$ with
$r=P^+(p-1)>x^{1/\log\log x}$ and $\Omega(a)< L$.
Since $p$ is symmetric,
there is some $d\mid a$ with at least one of $p+d$, $p-d$, $p+dr$, $p-dr$ prime.  Write $a=dm$.  For a given
pair $d,m$ with $dm< x^{1-1/\log\log x}$, let $R(x,d,m)$ denote the number of primes
$r\le x/dm$ with $dmr+1$ prime and at least one of $dmr+d+1$, $dmr-d+1$, $dmr+dr+1$, $dmr-dr+1$  prime.
  Again by Brun's method we have uniformly for $x$ large that
$$
R(x,d,m)\le \frac{x}{dm(\log x)^3}(\log\log x)^{O(1)}.
$$
It remains to sum this expression over pairs $d,m$ with $dm< x^{1-1/\log\log x}$ and $\Omega(dm)< L$.
Let $E$ denote the reciprocal sum of all primes and prime powers less than $x$.
We have
\begin{align*}
\sum_{\substack{dm<x^{1-1/\log\log x}\\\Omega(dm)<L}}\frac1{dm}&
\le\sum_{i+j< L}\sum_{\substack{d<x\\\omega(d)=i}}\frac1d
\sum_{\substack{m<x\\\omega(m)=j}}\frac1m\\
&\le\sum_{i+j< L}\frac1{i!}E^i\frac1{j!}E^j=\sum_{k< L}\frac1{k!}E^k\sum_{i+j=k}\frac{k!}{i!j!}\\
&=\sum_{k< L}\frac1{k!}(2E)^k\ll\frac1{L!}(2E)^L,
\end{align*}
since $E=\log\log x+O(1)$.  A
short calculation then shows that this expression is $(\log x)^{2-\eta}(\log\log x)^{O(1)}$.
Thus, the sum of $R(x,d,m)$ over pairs $d,m$ is at most $\pi(x)(\log x)^{-\eta}(\log\log x)^{O(1)}$, so
completing the proof.

\section{The proof of Theorem \ref{thm:main2}}

In her dissertation, Spiro \cite{S}
showed that the equation $d(n)=d(n+5040)$ has
infinitely many solutions, where $d(n)$ is the divisor function.
Heath-Brown \cite{HB} has shown that one can replace
5040 with 1 in this theorem, a key ingredient (see also
\cite{HBtoo,HBcorr}) being the existence of sets with
the property described in the next lemma (and another property
that is not needed here).

\begin{lemma}
\label{lem:one}
For every $k\ge 2$ there is a set $\cA_k\subset\N$ 
with $k$ elements such that $\gcd(a,b)=|a-b|$ for all $a,b\in\cA_k$, $a\ne b$.
\end{lemma}

An example of such a set when $k=4$ is $\{6,8,9,12\}$.

If we have two numbers $a<b$ with $\gcd(a,b)=b-a$ and an
integer $n$ for which $p=an+1$ and $q=bn+1$ are both prime,
then $\{p,q\}$ is a symmetric pair.
Thus, under the prime $k$-tuples conjecture we obtain
infinitely many symmetric pairs.
Alternatively, the prime $k$-tuples conjecture implies
the existence of infinitely many twin prime pairs $\{p,p+2\}$, which
are symmetric.

The preceding statements are still conjectural, but the Maynard-Tao theorem
gives us a path for producing infinitely many symmetric primes.

\begin{lemma}
\label{lem:two}
For every $m\ge 2$ there is a set $\cP_m$ of
$m$  primes  such that  for all $p,q\in\cP_m$, $p\ne q$, we have
 $\gcd(p-1,q-1)=|p-q|$.
Moreover, one can find such a set whose least element exceeds $m$.
\end{lemma}

\begin{proof}
Recall that a $k$-tuple of linear forms $\{g_it + h_i\}_{i=1}^k$
is said to be \emph{admissible} if the associated polynomial
$\prod_i(g_it + h_i)$
has no fixed prime divisor, i.e., for each prime $p$ there is an
integer $t$ with none of $g_it+h_i$ divisible by $p$. To prove the
lemma, we apply a remarkable theorem of Maynard
(see, e.g., \cite[Theorem 3.4]{M2}) and Tao (unpublished)
in the direction of the prime $k$-tuples conjecture.  

\begin{theorem}[Maynard--Tao]
\label{thm:maynard-tao}
For every $m \ge 2$  there is an integer  $k=k_m$,
depending only on $m$, such that
 if $\{g_it + h_i\}_{i=1}^k$ is admissible,
\begin{equation*}
 g_1,\ldots,g_k > 0,
  \qquad\text{and}\qquad
   {\textstyle \prod_{1 \le i < j\le k} } (g_ih_j - g_jh_i) \ne 0,
\end{equation*}
then  $\{g_in + h_i\}_{i=1}^k$  contains $m$ primes
for infinitely many $n \in \N$.  In fact, the number of such $n\le x$ is $\gg x/(\log x)^k$.
\end{theorem}

We apply Theorem~\ref{thm:maynard-tao} to the linear forms
$\{a_it+1\}_{i=1}^k$, where $k= k_m$ and the integers $a_i$
are the elements of a set $\cA_k$ of the type described in Lemma~\ref{lem:one}.
Then  $\{a_in+1\}_{i=1}^k$ contains $m$ primes
for infinitely many $n\in\N$, and Lemma~\ref{lem:two} follows at once.
\end{proof}

One can adapt the work of \cite{P1,P2} to show that $50$ is an acceptable value of $k_2$ in Theorem \ref{thm:maynard-tao}. This explains the ``49'' in the lower bound $S(x) \gg \pi(x)/(\log{x})^{49}$ claimed in the introduction.

To prove Theorem \ref{thm:main2}, which asserts that
some of the sets $\cP_m$ in Lemma~\ref{lem:two} consist of
\emph{consecutive} primes, we use the following result from
Banks~\emph{et al.}~\cite{BFTB}.

\begin{theorem}[Banks--Freiberg--Turnage-Butterbaugh]
 \label{thm:BFTB}
Let $m \ge 2$  and $k =k_m$, where $k_m$ is
as in the Maynard--Tao theorem.
Let   $b_1, \ldots, b_k$    be   distinct   integers   such   that
$\{t + b_j\}_{j=1}^k$  is admissible, and let  $g$ be an arbitrary positive
integer that is coprime to $b_1\cdots b_k$.
Then, for some subset
$
\{h_1,\ldots,h_m\} \subseteq \{b_1,\ldots,b_k\},
$
there    are    infinitely    many    $n  \in  \N$    such   that
$gn + h_1,\ldots,gn + h_m$ are consecutive primes.
\end{theorem}

Now, let $m\ge 2$ and $k\ge k_m$.  By Lemma~\ref{lem:two}
there exists a set of primes $\cP_k=\{b_1,\ldots,b_k\}$ such that
$\gcd(b_i-1,b_j-1)=|b_i-b_j|$ for all $1\le i<j\le k$, and each
$b_i$ exceeds $k$ (thus, the $k$-tuple $\{t + b_j\}_{j=1}^k$  is admissible).

Notice that $g=\prod_i(b_i-1)$ is coprime to $b_1\cdots b_k$. Otherwise, there are indices $i, j$, with $i\ne j$, for which $b_j \mid b_i-1$. But then $b_i \ge 2b_j+1$, and
\[ b_j + 1 \le b_i - b_j = \gcd(b_i-1,b_j-1) \le b_j-1, \]
which is absurd.

By Theorem~\ref{thm:BFTB} there is a subset
$\{h_1,\ldots,h_m\} \subseteq \{b_1,\ldots,b_k\}$ with the
property that $P_1=gn + h_1,\ldots,P_m=gn + h_m$ are consecutive primes
for infinitely many $n\in\N$.  Since $h_i-1\mid P_i-1$ for each $i$,
and
$$
|P_i-P_j|=|h_i-h_j|=\gcd(h_i-1,h_j-1)\qquad(1\le i<j\le m),
$$
it follows that $|P_i-P_j|=\gcd(P_i-1,P_j-1)$ when $i\ne j$, i.e.,
$\{P_i,P_j\}$ is a symmetric pair. This completes the proof of Theorem \ref{thm:main2}

\section{Computations}

In \cite{FLP} some values of $S(x)$ for $x$ up to the $10^5$th prime were given.  The data
did not strongly suggest that $S(x)=o(\pi(x))$;
 in fact, it seemed more plausible that $S(x)/\pi(x)\approx0.83$.
Using Mathematica we have extended the calculation to the $10^8$th prime and we see that $S(x)/\pi(x)$ continues
to be in no hurry to get to zero, but  progress towards this limit is somewhat discernible.   The descent to zero does
indeed appear to be not so different than the main term in our upper bound.

\begin{table}[h!]
  \begin{center}
    \caption{Tabulation of $S(p_n)$, the number of symmetric primes to the $n$th prime.}
    \label{tab:table1}
    \begin{tabular}{|l|c|c|c|} 
    \hline
      $n$ & $S(p_n)\vphantom{\Big)}$ & $S(p_n)/n$ & $1/(\log p_n)^\eta$\\
      \hline
     $ 10$ & 9 & 0.9000&0.9008\\
     $10^2$& 86 & 0.8600&0.8536\\
    $10^3$ & 864& 0.8640&0.8279\\
     $10^4$& 8473                &  0.8473 &0.8101 \\
    $10^5$& 83263 &0.8326&0.7964\\
   $10^6$ & 819848 & 0.8198&0.7854\\
    $10^7$ & 8098086& 0.8098&0.7761\\
      $10^8$&80112625&0.8011 &0.7681\\
      \hline
    \end{tabular}
  \end{center}
\end{table}


\section{Graph problems}

Consider a graph on the odd primes where two primes are connected by an edge if they form a symmetric pair.
The asymmetric primes are isolated nodes.  Must every connected component be finite?  At the
other extreme,
removing the asymmetric primes, is the graph connected?  If not, what is the least symmetric prime that
is not in the component containing the prime 3?  Does the graph have infinitely many components?
Does it contain a complete graph $K_m$ on $m$ vertices for every $m$?  The answer to this last question is ``yes", from Theorem \ref{thm:main2}. Clearly there cannot
exist an infinite complete subgraph since if $p<q$ are a symmetric pair, then $q<2p$.  Say a prime $p$ is
\emph{$m$-symmetric} if it is in a $K_m$.  It would be interesting to investigate the distribution
of $m$-symmetric primes; the number of them to $x$ is $\pi(x)/(\log x)^{O_m(1)}$, but what can be said about
the exponent here?
\bigskip

\noindent
{\bf Acknowledgment}.  We thank James Maynard for some helpful comments.

\bibliographystyle{amsplain}

\end{document}